\documentclass[10pt]{article}
\usepackage[utf8]{inputenc}

\usepackage{palatino}

\usepackage{amsmath, amssymb}

\usepackage{amsthm}

\usepackage{mdframed}

\usepackage{comment}

\usepackage{float}

\usepackage{appendix}

\usepackage{hyperref}

\newtheorem{theorem}{Theorem}[section]
\newtheorem{axiom}[theorem]{Axiom}
\newtheorem{conjecture}[theorem]{Conjecture}
\newtheorem{corollary}[theorem]{Corollary}
\newtheorem{definition}[theorem]{Definition}

\newtheorem{proposition}[theorem]{Proposition}
\newtheorem{prop}[theorem]{Proposition}

\newtheorem{question}[theorem]{Question}

\newcommand{\Ce}{{\mathcal C}}

\newcommand{\Pe}{{\mathbb P}}
\newcommand{\Qu}{{\mathbb Q}}

\renewcommand{\emptyset}{\varnothing}

\numberwithin{equation}{section}

\newcommand{\MA}{\mathrm{MA}}

\title{Steel's Programme: Evidential Framework, the Core and Ultimate-$L$}
\author{Joan Bagaria,\footnote{Universitat de Barcelona, ICREA. bagaria@ub.edu; joan.bagaria@icrea.cat} \ Claudio Ternullo\footnote{Universitat de Barcelona. claudio.ternullo@ub.edu.}}
\date{\today}

\begin{document}

\maketitle

\begin{abstract}

We address Steel's Programme to identify a `preferred' universe of set theory and the best axioms extending $\mathsf{ZFC}$ by using his multiverse axioms $\mathsf{MV}$ and the `core hypothesis'. In the first part, we examine the evidential framework for $\mathsf{MV}$, in particular the use of large cardinals and of `worlds' obtained through forcing to `represent' alternative extensions of $\mathsf{ZFC}$. In the second part, we address the existence and the possible features of the core of $\mathsf{MV}_T$ (where $T$ is $\mathsf{ZFC}$+Large Cardinals). In the last part, we discuss the hypothesis that the core is Ultimate-$L$, and examine whether and how, based on this fact, the Core Universist can justify $V$=Ultimate-$L$ as the best (and ultimate) extension of $\mathsf{ZFC}$. To this end, we take into account several strategies, and assess their prospects in the light of $\mathsf{MV}$'s evidential framework. 
    
\end{abstract}

\section{Introduction}

\subsection{Steel's Programme}

In \cite{steel2014}, John Steel presented a version of the set-theoretic multiverse consisting of `worlds' (\textit{set-generic extensions} of $V$), and a collection of axioms for it, $\mathsf{MV}$, and also made the hypothesis that such a multiverse might contain a \textit{core}, that is, a world included in all other worlds, which would act as `preferred universe', as the `real $V$'.\footnote{\cite{steel2014}, p. 167ff. Further discussion of Steel's $\mathsf{MV}$ axioms may be found in \cite{maddy-meadows2020} (for the discussion of the `core hypothesis', see, in particular, sections 5-6).}  A few years after the appearance of \cite{steel2014}, in the context of research on \textit{set-theoretic geology},\footnote{For a review of some set-theoretic geological concepts, see section 3.1.} it was proved (in \cite{U:ECM}) that, if there exists an \textit{extendible} cardinal, then $V$ has a smallest ground, and that such smallest ground is the $\kappa$-mantle of $V$ itself, where $\kappa$ is the least   extendible cardinal. A noticeable consequence of Usuba's result is that the \textit{multiverse} of the theory $\mathsf{MV}_T$, where $T=\mathsf{ZFC}+$`there exists a proper class of extendible cardinals' has a \textit{core}. Now, given the presuppositions of $\mathsf{MV}$, in particular, its reliance upon the whole hierarchy of Large Cardinals (LCs), it makes full sense to investigate the features of the multiverse of $\mathsf{MV}_T$, in particular, the features of its core. 

Steel's `core hypothesis' also has connections with another, recently arisen, fundamental set-theoretic hypothesis, that is, Woodin's `Ultimate-$L$ Conjecture'.\footnote{\label{Woodin}For the bulk of all results on Ultimate-$L$, see \cite{woodin2017}. For a more accessible exposition, see \cite{woodin2011b}, pp. 457-69. See also section 5 of the present paper.} The inner model programme has progressively unveiled the existence of `canonical' inner models of $\mathsf{ZFC}$+LCs; Woodin's Ultimate-$L$ Conjecture asserts that $\mathsf{ZFC}$+LCs also proves the existence of a \textit{weak extender model} for a cardinal $\delta$ being supercompact which also satisfies $V$=Ultimate-$L$.\footnote{\label{supercomp} An uncountable cardinal $\kappa$ is $\gamma$-supercompact, for some $\gamma \in Ord$,  if and only if there is an elementary embedding $j: V \rightarrow M$, with critical point $\kappa$, such that $M^{\gamma} \subseteq M$. A supercompact cardinal $\kappa$ is a cardinal that is $\gamma$-supercompact for all $\gamma \in Ord$.} In turn, Steel has made the hypothesis that Ultimate-$L$ might be the most `suitable' candidate as the \textit{core} of $\mathsf{MV}$.\footnote{\cite{steel2014}, p. 169ff., in particular, section 6. See also footnote \ref{AxiomH}, and \cite{maddy-meadows2020}, pp. 148-49.} 

Now, given $\mathsf{MV}_T$, where, again, $T=\mathsf{ZFC}$+`there exists a proper class of extendible cardinals', Steel's Programme, as we will call it, may be formulated as follows: 

\vspace{11pt}

\noindent
\textbf{Steel's Programme}. Use facts about the core of $\mathsf{MV}_T$ as evidence for the following claims:

\begin{enumerate}

\item $V$ is the core.

\item The core is Ultimate-$L$.
    
\item $\mathsf{ZFC}$+LCs+$V$=Ultimate-$L$ has a better claim than any other theory to be seen as the best (ultimate) extension of $\mathsf{ZFC}$.\footnote{\label{steel}It should be noted that Steel's Programme, although `implicit' in \cite{steel2014}, is not openly formulated as such by Steel himself.}

\end{enumerate}

\noindent
As we shall see in more detail later on, the execution (and meaningfulness) of Steel's Programme crucially depends on the relationship between $\mathcal{L}_\mathsf{MV}$, the multiverse language, and $\mathcal{L}_\mathsf{\in}$, the language of set theory, as expressed by a `translation function', which, in turn, shows that $\mathcal{L}_\mathsf{MV}$ is a \textit{sublanguage} of $\mathcal{L}_\mathsf{\in}$. As a consequence of this, the existence of the core of the multiverse of $\mathsf{MV}$ has bearings on 
set theory, as `standardly' construed as the theory of $V$, thus, Steel's Programme makes full sense. In broader philosophical terms, the interaction between the two `languages' may be construed as a way to provide set-theorists with a strategy to respond to the following question:

\begin{question}\label{mainquestion}

Is there a `preferred' universe of sets? 

\end{question} 

Clearly, if the Programme is successful, then the response to this question is in the positive, a fact which would have a considerable impact on our understanding of set theory, in particular, of its foundations. 

The main purpose of the paper is to provide a comprehensive assessment of the Programme, by focussing on three main topics: (1) its `evidential framework': in particular, the use of LCs and of forcing extensions of $V$ as `worlds' in the $\mathsf{MV}$ axiomatic set-up (section 2); (2) the core: the assumptions needed for its existence, and its level of (in)determinacy over $\mathsf{MV}$ (section 3 and 4); (3) the hypothesis that Ultimate-$L$ is the core, and the justification for this claim (section 5). 

But first, we would like to introduce very briefly the general features of the broader philosophical context in which Steel's Programme may be discussed.

\subsection{Two Kinds of Universism}

$\mathsf{MV}$ locates itself in the current universe/multiverse debate, and, consequently, in the debate on \textit{pluralism}, that is, on the issue of whether mathematical truth splits into \textit{many}, mutually incompatible, truths.\footnote{For an overview of the debate, see \cite{koellnerxxx}, and \cite{afht2015}.} Pluralism is very often taken to correspond to `ontological' pluralism, that is, to the view that:

\vspace{11pt}

\noindent
\textbf{(Ontological) Pluralism}.  There are many, \textit{alternative} universes of set theory (there is a set-theoretic \textit{multiverse}).\footnote{This viewpoint is most prominently represented by \cite{hamkins2012}.}  

\vspace{11pt}

\noindent
But note that some pluralists would just commit themselves to semantic pluralism, that is, to the view that the \textit{truth-value} of all set-theoretic statements undecidable from the $\mathsf{ZFC}$ axioms is indeterminate (neither true nor false). The opposite camp is represented by:

\vspace{11pt}

\noindent
\textbf{(Ontological) Non-Pluralism}. There is just \textit{one} universe of set theory.\footnote{For arguments against set-theoretic pluralism, see \cite{martin2001} and \cite{koellnerxxx}; for the potential skeptic consequences of multiversism, see \cite{button-walsh2018}, pp. 206-213, and \cite{barton2016}.}

\vspace{11pt}

\noindent
In what follows, we will mostly refer to the position above as:

\vspace{11pt}

\noindent
\textbf{Classic Universism}. Set theory is the theory of a \textit{single} universe, $V$, whose features are pinned down by the $\mathsf{ZFC}$ axioms (and, potentially, by other axioms extending $\mathsf{ZFC}$).\footnote{Although, as pointed out by a reviewer, it is consistent with the Universist position that the features of the `real $V$' might be, at least in part, unknowable.}  

\vspace{11pt}

\noindent
Steel himself introduces and considers several theses, of different nature and strength, concerning set-theoretic ontology and truth. The philosophical thesis which postulates the existence of the core of $\mathsf{MV}$'s multiverse is what Steel calls Weak Absolutism:

\vspace{11pt}

\noindent
\textbf{Weak Absolutism}. The multiverse has a core.\footnote{\cite{steel2014}, p. 168. On Steel's own classification, our Classic Universism corresponds to Strong Absolutism, whereas Ontological Pluralism, presumably, corresponds to Strong Relativism (which, however, isn't described by Steel). Weak Relativism is the thesis that the whole of $\mathcal{L}_\mathsf{\in}$ is `meaningful' in $\mathcal{L}_\mathsf{MV}$, a position which leaves it open whether there really exists a `reference universe' in $\mathcal{L}_\mathsf{MV}$.}

\vspace{11pt}

\noindent
Crucially for our purposes, Weak Absolutism quite naturally leads to the following stronger view:  %be formulated in alternative, and possibly more general, philosophical terms, as:

\vspace{11pt}

\noindent
\textbf{Core Universism}. Set theory is the theory of \textit{multiple} (set-theoretic) universes, that is, of a \textit{multiverse}, which also contains a \textit{core universe}. Such a universe has a better claim to be seen as the `ultimate universe of sets' than any other universe.

\vspace{11pt}

We believe that this position makes sense, especially in view of Steel's programme, insofar as, if the core of the multiverse of $\mathsf{MV}$ exists, then it makes sense to claim that $V$ is the core of such multiverse.

To state very quickly the main differences between Classic and Core Universism: the Classic Universist may be standardly characterised as someone believing that our \textit{intuition} of sets, or the `concept of set' itself, will provide us with a \textit{unique}, consistent extension of the $\mathsf{ZFC}$ axioms which will \textit{uniquely} fix the truth-value of the undecidable statements.\footnote{For a classic articulation of this position, see \cite{godel1947}.} %For an overview of alternative conceptions justifying the axioms, see \cite{incurvati2020}.} 
By contrast, the Core Universist may be characterised as someone who takes all alternative `universes' to be equally legitimate; however, such a Universist will also hold that each \textit{universe} (or, if you wish, \textit{theory}) contains `traces' of a single, `preferred' universe, and much of the value of the position consists in showing that the claim is true, that is, that a core universe is really \textit{detectable} within the multiverse itself. Of course, the Core Universist also expects to be able to describe the properties of the core in a satisfactory way.

In order to attain a reduction of set-theoretic incompleteness, the Classic Universist will suggest further exploration of our intuitions about sets, or sharpening of the concept of set, whereas the Core Universist will suggest further exploration of the properties of the core through the multiverse axioms. 

Now, it is clear that Steel's Programme advocates the Core Universist's standpoint, and, therefore, has deep implications on the preferability of Core over Classic Universism: as already said, the Programme, if successful, would lend support to Core Universism. Indeed, the Core Universist's construal of the Programme's goals and results could be condensed as follows: `Non-pluralism about set-theoretic ontology cannot be correct, as we are aware of the existence of many alternative universes (as well as of alternative theories extending $\mathsf{ZFC}$). However, given a suitable version of the multiverse, one resting upon significant bits of current set-theoretic practice (upon the `evidential framework' addressed in the next section), we may identify a `preferred' universe within the multiverse itself. But then, for all our foundational purposes, such a universe might be seen as a fully adequate instantiation of our pre-theoretic notion of `single universe', perhaps not exactly the Classic Universist's one, indeed, a more `pragmatic', but equally justified, version of it.' 

The Weak Absolutist's position has already been vindicated by Usuba's result: if there are sufficiently strong LCs, then $\mathsf{MV}$ has a core. However, whether, and in what sense, the core could be seen as the `ultimate' (`preferred') universe of sets, as claimed by the Core Universist, namely as the equivalent of the Classic Universist's $V$, is still open to debate, as we shall see in the next sections.

\section{The Evidential Framework of $\mathsf{MV}$}

In the preliminaries of \cite{steel2014}, Steel introduces and advocates a few distinctive positions, which he takes to be the main motivation for, and underlying conceptual framework of, his multiverse conception, as embodied by the $\mathsf{MV}$ axioms. We summarise them below: 

\begin{enumerate}

\item Large Cardinals are `practically' \textit{necessary} to extend $\mathsf{ZFC}$, mostly as a consequence of the following two phenomena:

\begin{enumerate}

\item LCs `calibrate' the consistency of extensions of $\mathsf{ZFC}$, and of  \emph{most} undecidable set-theoretic statements;

\item LCs maximise \textit{interpretative power}, insofar as the hierarchy of consistency strengths of Large Cardinal Axioms (LCAs) is very aptly matched by the hierarchy of proof-theoretic strengths of extensions of $\mathsf{ZFC}$.\footnote{\label{determinacy}\cite{steel2014}, p. 162, also points out that additional confirmation of the `correctness' of LCs is provided by the multiple correspondences between these and Determinacy Axioms; in turn, the successfulness of Determinacy Axioms would be demonstrated by their usefulness and fecundity in many areas of mathematical research. For an overview of the connections between Determinacy and LCs, see \cite{maddy1988b} and \cite{koellner2014}.}

\end{enumerate}

\item Set theory should fundamentally be seen as the theory of the \emph{forcing extensions} and \emph{inner models} of models of the theory: $\mathsf{ZFC}+$LCAs. This position, in turn, rests upon the fact that independence proofs are practically carried out in the model theory of (fragments of) $\mathsf{ZFC}+$LCAs and, most crucially for Steel's purposes, also upon the fact that all `natural' theories extending $\mathsf{ZFC}$ may be mutually connected through using models with LCs.\footnote{It should be noted that Steel's theory leaves out many inner models which do not satisfy strong LCs: e.g., since $L$ does not have those cardinals, no model satisfying $\mathsf{ZFC}+$V=L can be a world of $\mathsf{MV}$. We thank an anonymous reviewer for raising this point. However, this does not mean that these models (and theories) cannot be defined inside any of $\mathsf{MV}$'s worlds: see section 2.1.}

\item Set theory, as currently practised and interpreted, splits into several `natural' theories, all of which extend $\mathsf{ZFC}$. Part of the rationale for adopting set theory as a foundation of mathematics, thus, consists in describing all such theories and their connections through exploiting, in turn, connections between LCs and models containing them.\footnote{\label{steel} The point is made by Steel in section 5 of \cite{steel2014}, p. 164. However, in private correspondence with the authors (26/10/2020), Steel has understated this fact as a motivation for $\mathsf{MV}$, and placed a lot more emphasis on the existence of the translation function from $\mathcal{L}_{\mathsf{MV}}$ to $\mathcal{L}_{\in}$ as the main rationale for formulating $\mathsf{MV}$.}

\item In order to describe all such theories, one should make use of \textit{forcing}. To be sure, models obtained through forcing have become an essential tool to produce a wide range of set-theoretic `universes' upholding or violating set-theoretic principles, and, in particular,  models obtained through forcing are used to prove the \textit{equiconsistency} of theories with LCs. So, in practice, \textit{forcing extensions} of the universe should be taken to stand for different universes (Steel's \textit{worlds}).\footnote{However, notice that forcing extensions of $V$ may produce violations of LCs. So, if one wants to preserve LCs uniformly, one should rule out several models obtained through different kinds of \textit{class forcing}. Cf. \cite{steel2014}, p. 167.} 

\end{enumerate}

In the next subsections, we wish to discuss at length salient aspects of the points above and, through this, provide an assessment of the evidence for the $\mathsf{MV}$ axioms.

\subsection{Natural Theories, Large Cardinals, and Worlds}\label{2.1}

As stated in bullet point (3.) above, one of the purposes of Steel's multiverse conception is precisely that of `representing' all `natural' theories extending $\mathsf{ZFC}$ within a \textit{unified} axiomatic framework, without explicitly and directly having to deal with `universes'.  

However, in order to show that the axioms are not \textit{semantically} vacuous, there will have to be objects fixing their interpretation, and such objects are, on the one hand, \textit{sets} and, on the other, \textit{worlds}\footnote{\cite{steel2014}, p. 165.}. But notice the peculiar use of worlds in $\mathsf{MV}$: the latter are introduced, and described, to account for the \textit{representability} of \textit{different} theories.\footnote{Both \cite{maddy2017}, pp. 310ff., and \cite{maddy-meadows2020}, pp. 123-4 stress the importance of this fact, by contrasting Steel's conception with those of Hamkins and Woodin, which would, on the contrary, openly commit to significantly more \textit{ontological} forms of multiversism, and to a `metaphysics of universes' from the beginning.}  The strategy is, crucially, facilitated by both proof-theoretic and model-theoretic facts concerning LCs, in particular, the following conjecture:

\begin{conjecture}\label{Con1} 

Any `natural' extension of $\mathsf{ZFC}$ is either equiconsistent with $\mathsf{ZFC}$ or equiconsistent with $\mathsf{ZFC}+A$, where $A$ is a LCA. Moreover, the consistency strengths of `natural' extensions of $\mathsf{ZFC}$ are well-ordered.\footnote{\label{nat}The qualifier `natural' in the statement of Conjecture \ref{Con1} is necessary, for there are (contrived, hence `unnatural') examples of consistent sentences $\varphi$ whose consistency strength is strictly stronger than $\mathsf{ZFC}$ (meaning that $\mathsf{ZFC}$ does not prove that $Con(\mathsf{ZFC})$ implies $Con(\mathsf{ZFC} + \varphi)$), yet $\mathsf{ZFC} + \varphi$ does not even prove the consistency of $\mathsf{ZFC} + Con(\mathsf{ZFC})$. Hence the consistency of $\mathsf{ZFC} + \varphi$ does not even yield the consistency of $\mathsf{ZFC}+$`There exists an inaccessible cardinal'. }

\end{conjecture}

The main upshot of Conjecture \ref{Con1} is that, since all theories $\mathsf{ZFC}$+LCA considered so far are arranged in a well-ordered scale of consistency strengths,\footnote{Pending open questions about comparability or equivalence of the consistency of certain large-cardinal notions (e.g., the equiconsistency of \textit{supercompact} and \textit{strongly compact cardinals}), which presumably will eventually be resolved.} then also all natural theories are. Thus, the `invisible web' binding together all natural theories is, finally, cast into sharp relief by the proof-theoretic connections among all LCAs. 

The next step is to exploit the fact that, using LCs, one is able to construct models, in particular \textit{set-forcing extensions} and \textit{inner models}, which satisfy any natural theory. Using this fact, and Conjecture \ref{Con1}, Steel is, therefore, able to formulate, in full, the main meta-theoretic constraint presiding over his `multiverse' (bullet point (2.) above):

\vspace{11pt}

\noindent
\textbf{Meta-Theoretic Constraint}. $\mathsf{MV}$'s worlds are just those models of $\mathsf{ZFC}$+LCA which are needed to incorporate all natural theories extending $\mathsf{ZFC}$.

\vspace{11pt}

Finally, since inner models may well be defined inside forcing extensions, what we just need to characterise worlds are LCs and forcing: the $\mathsf{MV}$ axioms reflect this state of affairs, by quantifying over set-forcing extensions (and their grounds) with LCs. 

We will deal with the models of $\mathsf{MV}$ later on, but, first, we wish to express several concerns about Steel's use of, and primary reliance upon, LCs, and about the notion of `naturalness'.  

The first qualm refers to Conjecture \ref{Con1}, that all `natural' extensions of $\mathsf{ZFC}$ are equiconsistent with either $\mathsf{ZFC}$ or some LCA. While no evidence against it has been found thus far, the conjecture is far from being settled, and things are compounded by the fact that we lack a general definition of `large cardinal'. Now, it would certainly be unfair to have the \textit{unsharpness} of the notion count against Steel's project; however, the absence of ultimate evidence that any undecidable statement will turn out to be equiconsistent with a LCA is a fact, which casts some doubts on the tenability of the conjecture. 

A second, possibly more malignant, issue is that the notion of `natural theory' is unclear. What we are told by Steel, at the very outset, is that a set-theoretic statement does qualify as `natural' if it is consistent with $\mathsf{ZFC}$, asserts some `facts' about sets, and is not of a metamathematical or proof-theoretic nature, but this is really not much.\footnote{Cf. \cite{steel2014}, p. 157: `By `natural' we mean considered by set theorists, because they had some set-theoretic idea behind them.'} Moreover, $\mathsf{MV}$'s scope for `naturalness' is too restrictive from the beginning, as it leaves out theories, such as $\mathsf{ZF}+$AD, which unquestionably express deep set-theoretic facts.

%theories with Choice through using forcing extensions  has emerged, which seems to be even  alternative to $\mathsf{ZFC}$ has emerged  emerged  that other theories do not enjoy, but this is provably false: theories without Choice are \textit{equiconsistent} with LCs, may incorporate strong set-theoretic statements (such as the Axiom of Determinacy),\footnote{See footnote \ref{determinacy}.}  

As an interpretative option, one could define a set-theoretic statement $\varphi$ to be `natural' if $\mathsf{ZFC}$ plus some LCA proves that $\varphi$ holds in a forcing extension of $V$, or in some definable (allowing for set parameters) inner model of $\mathsf{ZFC}$ of a forcing extension of $V$ (including definable, with set parameters, class forcing extensions that preserve $\mathsf{ZFC}$). By this interpretation, all $\mathsf{ZFC}$ axioms as well as all LCAs are natural, and so are CH, V=L, V=HOD, SCH, as well as their negations (in fact, all \textit{genuinely} set-theoretic statements known to be consistent with $\mathsf{ZFC}$ and asserting some facts about sets). Moreover, in this case, even theories contradicting Choice, like $\mathsf{ZF}$+AD, would now be seen as `natural', as they are also equiconsistent with $\mathsf{ZFC}$+LCAs.\footnote{However, see the considerations in fn. \ref{nat}; of course, also this interpretation of `naturalness' will have to rule out such statements as $\mathsf{ZFC}+Con(\mathsf{ZFC})$.}  

This interpretation seems to fit Steel's goals, but, surely, `equiconsistency with $\mathsf{ZFC}$ or $\mathsf{ZFC}$+LCAs' does not seem to square very well with the common-sense, intuitive meaning of `naturalness'. Moreover, one could easily reverse this approach, consider a theory natural if it can be expressed in models of $\mathsf{ZF}$+LCAs, which also allow for the existence of even stronger LCs, such as Reinhardt and Berkeley Cardinals, and `incorporate' theories with Choice in models of $\mathsf{ZF}$+LCAs.\footnote{\label{LCBC}For Reinhardt and Berkeley cardinals, see \cite{bagaria-koellner-woodin2019}, sections 2-3, pp. 289-296. It should be noted that it is presently not known whether Choiceless Large Cardinals are consistent with $\mathsf{ZF}$. If consistent, these LCs would have a dramatic impact on other set-theoretic hypotheses, such as the Ultimate-$L$ Conjecture, and the HOD Dichotomy Hypothesis. In particular, they would imply that HOD is `far' from $V$, so, in a sense, they may also be taken to have noticeable \textit{maximising} virtues (so long as $V$=HOD is construed as a limiting hypothesis). Reinhardt cardinals may be seen as instantiating a different form of maximality, as they imply a `maximisation' of the level of resemblance of $V$ with itself in the non-trivial embedding $j: V \rightarrow V$ which is used for their definition. For further details and philosophical considerations on all these issues see, again, \cite{bagaria-koellner-woodin2019}, section 8, p. 309ff.}

%Of course, Steel's exclusive focus on `natural extensions of $\mathsf{ZFC}$' is deliberate; but, then, we are left with no indication whatsoever about why, in the light of the other bits of evidence for $\mathsf{MV}$, $\mathsf{ZFC}$ is preferable over other axiomatisations.

\subsection{Large Cardinals as Maximality Principles}

Steel's preference for LCs is also motivated by another principle, that he calls `maximise interpretative power', which might be seen as consisting of two parts:

\vspace{11pt}

\noindent
\textbf{Maximise Interpretative Power [MIP]}. (A.) The $\mathsf{MV}$ axioms should be able to `represent' as many theories (`natural extensions of $\mathsf{ZFC}$') as possible; (B.) all the theories represented by the $\mathsf{MV}$ axioms should be such that, for any two of them, $T$ and $S$, if $Con(T) \rightarrow Con(S)$, then $\Gamma_S \subseteq \Gamma_T$ (where, given a theory $T$, $\Gamma_T=\{\phi: T \vdash \phi\}$).\footnote{Cf. \cite{steel2014}, pp. 158-9; 165.} 

\vspace{11pt}

\noindent
LCAs constitute a paradigmatic case study with respect to \textbf{MIP}. As far as (A.) is concerned, we have seen that each `natural' theory $T$ is \textit{satisfied} in a forcing extension or inner model of another `natural' theory $S$, provided both $T$ and $S$ are equiconsistent with LCAs; as regards (B.), we know that the amount of mutual \textit{interpretability} of natural theories rises in proportion with their consistency strength. Thus, in the end, it would be legitimate to expect that:

\vspace{11pt}

\noindent
(*) As natural theories proceed up the large cardinal hierarchy in consistency strength, they agree on an ever-increasing class of mathematical statements.\footnote{Cf. \cite{maddy-meadows2020}, p. 128, from which (*) has been verbatim reproduced.} 

\vspace{11pt}

\noindent
A few observations are in order. The first one, concerning (B.), is that, so far, \textbf{MIP} has been shown to be satisfied by LCAs only partially, that is, for specific kinds of sentences in the L\'evy hierarchy of arithmetical sentences. The following is the most one can hope to prove so far:

\vspace{11pt}

\noindent
\textbf{Empirical Fact.}
For any two natural theories $T, S$ whose consistency strength is at least that of the theory: $\mathsf{ZFC}$+`there exist infinitely many Woodin cardinals', such that $Con(T) \rightarrow Con(S)$, we have that $({\Pi^{1}_\omega})_S \subseteq ({\Pi^{1}_\omega})_T$ (where, given a theory $T$, $({\Pi^{1}_\omega})_T$ is the set of $\Pi^{1}_{\omega}$ sentences provable in $T$).  

\vspace{11pt}

As a consequence, the applicability of \textbf{MIP} to LCAs has only been verified up to the level of second-order arithmetic.  

The second observation is as follows. Steel's purpose is that of \textit{representing} a multiplicity of theories, all of which extend $\mathsf{ZFC}$+LCs, and we wonder whether \textbf{MIP} is really compatible with this goal. For suppose (B.) were applicable to \textit{all} sentences in the L\'evy hierarchy; then, clearly, all the theories targeted by (A.), would not, just, be \textit{represented}, but, for all practical purposes, they would rather be \textit{amalgamated} into just one theory.\footnote{That \textbf{MIP} may really extend beyond $\Pi^{1}_{\omega}$ sentences, as the consistency strength of LCAs further increases, is very speculative (that it may extend to \textit{all} sentences of the L\'evy hierarchy is even less plausible). In particular, as pointed out by an anonymous reviewer, it is very doubtful that it might extend to the level of $\Sigma^{2}_{1}$ statements, that is, to the level of CH, although Steel conjectures that this could be the case, cf. \cite{steel2014}, p. 163.} One way out of this difficulty would be to see (A.) as being sanctioned by the presently limited range of applicability of (B.), but this wouldn't help fully ease the tension between (A.) and (B.).

Finally, although this is very speculative, the hierarchy of LCs contradicting Choice might, potentially, be more successful at instantiating \textbf{MIP} than the hierarchy of LCs with Choice, but, as we have seen, none of these theories features among those targeted by $\mathsf{MV}$.\footnote{For other forms of maximality Choiceless LCs may potentially embody, see the considerations in footnote \ref{LCBC}.}

\subsection{Models for $\mathsf{MV}$}\label{MV}

A major asset of $\mathsf{MV}$ is that this theory is complete with respect to a specific class of models $M^{G}$, with which we will deal in a moment, that is, one has that:
\[\mathsf{MV} \vdash \phi \leftrightarrow (\forall M^{G}) \ M^{G} \models \phi \]  We will not delve into the philosophical reasons for preferring a \textit{complete} axiomatisation of the multiverse over one which is not, as the task has already been carried out satisfactorily.\footnote{\label{mvcomplete}\cite{maddy-meadows2020}, pp. 137ff. The proof of the completeness of $\mathsf{MV}$ (Theorem 8) is on p. 158.} In this subsection, we would rather like to focus on the semantics of $\mathsf{MV}$, in particular, on its `natural' models, $M^G$,\footnote{See further, paragraph below the enunciation of the $\mathsf{MV}$ axioms.}
%The models $M^G$ are  class-generic extensions of a countable model $M$ by a generic $G$ for the class-forcing $Coll(\omega , <Ord)^M$ that collapses all ordinals to $\omega$. These models provide a complete semantics for $\mathsf{MV}$ (see \cite{maddy-meadows2020}).} 
and bring to light a slightly different mathematical approach to it (Proposition \ref{homo}).   

We start with reviewing, very quickly, the axioms. The language of $\mathsf{MV}$ is the first-order language of set theory with two sorts, namely \emph{Set} and \emph{World}. We introduce a minor tweak to Steel's original formulation of the axioms. As we already know, $\mathsf{MV}$ has, as its own base, the axioms of $\mathsf{ZFC}$ plus LCs (we shall mostly refer to the \textit{base theory} as $T$). Now, let a $T$-theory be a theory extending $\mathsf{ZFC}$ which is preserved by set-forcing extensions, and by going to set-forcing grounds: it turns out that any theory of the form $\mathsf{ZFC}$+`there exists a proper class of some kind of LCs' is a $T$-theory. Then, the $\mathsf{MV}$ axioms for $T$ $(\mathsf{MV}_T)$, which we will be mostly referring to and using throughout the paper, are the following ones:

\begin{enumerate}
\item (Extensionality for Worlds) If two worlds have the same sets, then they are equal.
\item Every world is a model of $T$.
\item Every world is a transitive proper class. An object is a set if and only if it belongs to some world. All worlds have the same ordinals.
\item If $W$ is a world and $\Pe \in W$ is a poset, then there is a world of the form $W[G]$, where $G$ is $\Pe$-generic over $W$.
\item If $U$ is a world and $U=W[G]$, where $G$ is $\Pe$-generic over $W$, then $W$ is also a world.
\item (Amalgamation) If $U$ and $W$ are worlds, then there are posets $\Pe\in U$ and $\Qu\in W$, and sets $G$ and $H$ $\Pe$-generic and $\Qu$-generic over $U$ and $W$, respectively,  such that $U[G]=W[H]$.
\end{enumerate}

If $M$ is a countable model of $T$ and $G$ is $Coll(\omega, <Ord)^M$-generic over $M$, let $M^G$ be the model whose sets are those in $M[G]$ (in the case $M$ is ill-founded, then  $M[G]$ is defined accordingly) and whose worlds are the grounds of models of the form $M[G\restriction \alpha]$, for some $\alpha \in Ord^M$. It can be easily shown that $M^G$ is a model of $\mathsf{MV}_{T}$, when $T=\mathsf{ZFC}$.\footnote{\cite{maddy-meadows2020}, Theorem 26, p. 155.}  Moreover, if $T$ is obtained by adding to $\mathsf{ZFC}$ axioms such as `There is a proper class of $P$-cardinals', where $P$ stands for any of the usual large-cardinals properties, then it can also be proved that $M^G$ is a model of $\mathsf{MV}_{T}$.%

The collection of all models $M^G$ provides a complete semantics, and the axiom which guarantees the completeness of $\mathsf{MV}_T$ is, as has been shown by \cite{maddy-meadows2020}, p. 134, Amalgamation. But, as stressed by the authors, a consequence of this fact is that, in any model of the form $M^G$, not all generic filters for posets in $M$ may be taken into account to produce \textit{forcing extensions} which act, as required by the Meta-Theoretic Constraint of section \ref{2.1}, as the 'worlds' of $\mathsf{MV}$, but only those which are produced by the $Coll(\omega, Ord^M)$-generic filter $G$ over $M$. 

We now proceed to prove the following:

\begin{proposition}
\label{homo}
The $\mathsf{MV}_T$ axioms imply that the multiverse is a homogeneous class-forcing  extension of each of its worlds.
\end{proposition}

\begin{proof} Since the models $M^G$ give a complete semantics for $\mathsf{MV}_T$, we may assume that every model of $\mathsf{MV}_T$ is of this form. Now, working in a model $M^G$ as above, let $W$ be a world. So  $W$ is a ground of a model  of the form $M[G\restriction \alpha]$, for some $\alpha \in Ord^M$.

Let  $\Pe\in W$ be a poset such that for some $H_0$ $\Pe$-generic over $W$, $W[H_0]=M[G\restriction \alpha]$, and let $\kappa \geq \alpha$ be an uncountable $W[H_0]$-cardinal such that the cardinality of $\Pe$, as computed in $W$, is less than $\kappa$. Let $H_1$  be $Coll(\omega , <\kappa +1) /G\restriction \alpha$-generic over $M[G\restriction \alpha]$ so that $M[G\restriction \alpha][H_1]=M[G\restriction \kappa +1]$. Thus, $W[H_0][H_1]=M[G\restriction \kappa +1]$ is a generic extension of $W$ by a poset of cardinality $\kappa$ that collapses $\kappa$ to $\omega$, hence by Kripke's theorem (\cite{jech2003}, Lemma 26.7) equivalent to $Coll(\omega , \kappa)$.  Since $Coll(\omega , \kappa)$ is homogeneous, the proposition follows. 
\end{proof}

\noindent
The following straightforward consequence of the $\mathsf{MV}_T$ axioms asserts that every set-forcing generic extension of a world is also a world.

\begin{proposition}
\label{coro1}
The $\mathsf{MV}_T$ axioms imply that if $W$ is a world and $G$ is $\Pe$-generic over $W$ for some poset $\Pe \in W$, then $W[G]$ is also a world.
\end{proposition}

\begin{proof}
Suppose $W$ is a world and $G$ is  $\mathbb{P}$-generic over $W$, for some poset $\mathbb{P}\in W$. By Axiom 3, let $W'$ be a world such that $G\in W'$. By Amalgamation (Axiom 6), let $W''$ be a world such that both $W$ and $W'$ are grounds of $W''$. As $W\subseteq   W[G]\subseteq W''$, and $W$ is a ground of $W''$, the forcing extension $W[G]$ of $W$ is also a ground of $W''$. Hence, since every ground of a world is also a world (Axiom 5), $W[G]$ is a world, as wanted.
%Fix a world $W$ and a poset $\mathbb{P}\in W$. By Axiom 4, there is some $\mathbb{P}$-generic filter $G_0$ over $W$ such that $W[G_0]$ is a world. So, there exists $p\in \mathbb{P}$ that forces the statement `$W[\dot{G}]$ is a world', where $\dot{G}$ is the standard $\mathbb{P}$-name for the generic filter. Suppose, for a contradiction, that some condition $q\in \mathbb{P}$ forces the statement `$W[\dot{G}]$ is not a world'. By the lemma, let $\mathbb{Q}$ be a homogeneous collapse forcing in $W$ such that  $\mathbb{P}$ completely embeds into $\mathbb{Q}$, so we may as well assume that $\mathbb{P}$ is a complete suborder of $\mathbb{Q}$. Since the collapse forcing $\mathbb{Q}$ is homogeneous, there is an automorphism of $\mathbb{Q}$ sending $p$ to $q$ that fixes $\mathbb{P}$ (\cite{jech2003}, Lemma 26.9), which is impossible, as $p$ and $q$ force contradictory statements.
\end{proof}

\subsection{$M^{G}$ and The Translation Function}

As already anticipated in section 1.1, Steel's conception rests upon the crucial fact that the language of $\mathsf{MV}$, $\mathcal{L}_\mathsf{MV}$, may be seen as a sublanguage of $\mathcal{L}_\mathsf{\in}$. In order to show this, Steel defines a recursive translation function $t$ from $\mathcal{L}_{\mathsf{MV}}$ into $\mathcal{L}_{\in}$, such that:

\[\mathsf{MV}_T\vdash \varphi \quad \mbox{ if and only if } \quad T\vdash t(\varphi) \tag{Transl} \]

\noindent
where $T$, as said, is $\mathsf{ZFC}$+LCs. The translation function may  be rendered more transparently in terms of the semantics of $\mathsf{MV}$ as follows:

\begin{theorem}[Translation Function]\label{Transl}

For any sentence $\varphi$ of $\mathcal{L}_{\mathsf{MV}}$, every countable model $M$ of $\mathsf{ZFC}$ and every $G \ Coll(\omega, <Ord^M)$-generic over $M$ 

\[M^G\models \varphi \quad \mbox{ iff } \quad M\models t(\varphi) \tag{Transl}\footnote{\cite{steel2014}, p. 166. For further details on (Transl), see \cite{maddy-meadows2020}, section 5, pp. 137-43.}\]

\end{theorem}

\noindent
It should be noted that, in the presence of the specified semantics of $\mathsf{MV}$, the sentence $t(\varphi)$ asserts `$\varphi$ is true in all multiverses obtained from me'. Equivalently, `$\varphi$ is true in some multiverse obtained from me'. Now, as said (section 2.3), Amalgamation is needed to show that the models $M^G$ provide a complete semantics for $\mathsf{MV}_T$. 

%It would, then, be natural to ask the following question:

%\begin{question}

%Is it possible to define a `translation function' in the context of multiverse axioms which do not imply Amalgamation?

%\end{question}

%\noindent
%The answer is probably no; more details are given in Appendix B.

\section{The Core of $\mathsf{MV}$}

In this section, we will, on the one hand, show how the `core hypothesis' arises in the context of the $\mathsf{MV}$ axioms, in particular, how it is mathematically justified; this will help us provide an answer to the following question:

\begin{question}

Under what circumstances is the \textit{core} definable, and how?

\end{question}

On the other hand, we will show why Usuba's result (Theorem \ref{usuba} below) implies that $\mathsf{MV}_T$, where $T=\mathsf{ZFC}$+`there exists a proper class of extendible cardinals', proves that there exists a \textit{core universe}. 

However, before proceeding to review all such results, we address the `core hypothesis' in the context of \textit{just} $\mathsf{ZFC}$, and recall relevant set-theoretic geological concepts and theorems which make sense of the hypothesis, and which will be instrumental to present our own results in section 4.

\subsection{Prelude: The (Outer) Core of $\mathsf{ZFC}$}

Initiated by Reitz and Hamkins a few years ago, set-theoretic geology has brought forward a very innovative approach to `universes of set theory' producible using forcing.\footnote{Fundamental references here are \cite{reitz2007}, and \cite{hamkins-reitz-woodin2008}. A comprehensive account of the programme's conception and results is in \cite{fuchs-hamkins-reitz2015}, but see also \cite{hamkins2012}, pp. 443-47, for further philosophical discussion.} Here follows a very brief review of its fundamental concepts.

Take $V$ to be a model of the axioms $\mathsf{ZFC}$. Hypothetically, it might be that $V$ is a forcing extension of a ground model $W$, that is, that there exists a $W$-generic filter $G \subseteq \mathbb{P} \in W$ such that $V=W[G]$. If this is the case, then it makes sense to explore the 
`geology' of $V$, that is, the collection of grounds, the generic extensions of grounds and their grounds, and so on, which $V$ might contain, where a ground of a model $M$ is a model $N$ of $\mathsf{ZFC}$ such that $M$ has been obtained through forcing over $N$, that is, as an extension $M=N[G]$, where $G$ is a generic filter of a partial order $\mathbb{P} \in N$. Further `geological' notions will then crop up, which can be spelt out as follows. 

The \textit{mantle} $\mathbb{M}$ is the intersection of all the \textit{grounds} of a model $M$ of $\mathsf{ZFC}$, and the \textit{bedrock} is a class $W$ which is a \textit{minimal} ground of $V$. By results of Usuba \cite{U:DDG}, the mantle is a model of $\mathsf{ZFC}$. Since \cite{reitz2007}, the following axiom has proved to be central to all geological investigations:

\begin{axiom}[Ground Axiom (GA)]

$V$ has no proper grounds. 

\end{axiom}

Now, if the mantle itself satisfies GA, then the mantle is a \textit{bedrock}, in particular, a \textit{minimum} bedrock contained in all other grounds, and the latter could legitimately be seen as the \textit{core universe} of the set-generic multiverse (generated over it) that we'll be describing in full detail in the next subsection.\footnote{\label{DDG}An essential property to be satisfied for this to take place is Downwards Directedness, that is, that any two grounds have a ground \textit{in common}.  The Downwards Directedness Hypothesis (DDG) is, precisely, the claim that all grounds are \textit{downwards directed}. \cite{U:DDG} proves that DDG already holds in $\mathsf{ZFC}$ (Proposition 5.1, p. 13).} A fuller mathematical characterisation of the core may, thus, be attempted. However, already at this point, it emerges that the features of the core may be (over $\mathsf{ZFC}$) indeterminate. This is shown by the following, fundamental theorem:

\begin{theorem}[Fuchs, Hamkins, Reitz]\label{mantle}

Every countable model of $\mathsf{ZFC}$ can be the mantle of another model of $\mathsf{ZFC}$.\footnote{\cite{fuchs-hamkins-reitz2015}, p. 464, Main Theorem 4.}

\end{theorem}

In an attempt to both attain a more determinate core, and further investigate its nature, set-theoretic geologists have taken into account a different hypothesis, which can be expounded as follows. If the mantle does not satisfy GA, then it makes sense to take into account the `mantle of the mantle',
$\mathbb{M}^1=\mathbb{M}^{\mathbb{M}}$, and then, always under the assumption that GA is not satisfied, the `mantle of the mantle of the mantle' and so on. In other terms, it makes sense to take into account the iteration of the `mantle operation'. Now, iterating the mantle is not a trivial task, as there are technical aspects involved which may prevent one from even \textit{defining} the \textit{n}th iterates of the mantle.\footnote{See \cite{fuchs-hamkins-reitz2015}, p. 496ff.} Assuming these difficulties may be overcome, one might ask whether, for an $\alpha \in Ord$ iterate of the mantle, one has that $\mathbb{M}^{\alpha}=\mathbb{M}^{\beta}$, for all $\beta>\alpha$, that is, whether there is a \textit{minimal} $\alpha$ such that $\mathbb{M}^{\alpha} \models GA$. If there is one, then $\mathbb{M}^{\alpha}$ is said to be the \textit{outer core} of the initial model (of $V$, if $V$ was such model). However, in this case, as in the previous one where $\mathbb{M} \models GA$, Theorem \ref{mantle} will imply that also the outer core does not have \textit{determinate} features, namely, that it might satisfy a wide range of mutually incompatible properties. 

Moreover, it has been shown that the outer core cannot be uniquely pinned down in the iteration process, as a recent result by Reitz and Williams has confirmed \cite{fuchs-hamkins-reitz2015}'s Conjecture 74,\footnote{Cf. \cite{reitz-williams2019}, p. 6, Theorem 3.1 and ff. Conjecture 74 is in \cite{fuchs-hamkins-reitz2015}, p. 497.} and shown that, for any $\alpha \in Ord$, the outer core of a model $M$ could be any $\mathbb{M}^{\alpha}$, that is, any $\alpha$-th iterate of the mantle of $M$ (possibly also including including $\mathbb{M}^{Ord}$, although no proof is currently available for this specific case).\footnote{\cite{fuchs-hamkins-reitz2015} even envisages a possible extension, over the theory $\mathsf{GBC}$, of the iteration of the mantle operation beyond $\mathbb{M}^{Ord}$ for some model $M$, something which would imply that $M$ has \textit{no} outer core at all. See \cite{fuchs-hamkins-reitz2015}, pp. 497-498.}

To sum up, $\mathsf{ZFC}$ alone does not, in itself, guarantee the existence of a `core universe' and, moreover, over $\mathsf{ZFC}$, no full-fledged, definite mathematical characterisation of the core may arise. As we shall see, a slightly different scenario comes to light, if one takes onboard LCs. 

\subsection{The Proof of the Existence}

\cite{steel2014} credits Woodin for observing that `if the multiverse has a definable world, then it has a unique definable world, and this world is included in all the others'.\footnote{\cite{steel2014}, p. 168.} We give next the first complete proof of this fact.\footnote{We thank John Steel for providing us, in private correspondence, with some details for the proof of Theorem \ref{defworld}.}

\begin{theorem}
\label{defworld}
If the multiverse has a definable world, then it has a unique definable world. More precisely, suppose $\varphi$ and $\psi$ are formulas in the language of the multiverse with only one free variable for sets. Then, 
$$\mathsf{MV}_T\vdash '\forall U,W (\forall x ((x\in U \leftrightarrow \varphi(x)) \wedge (x\in W \leftrightarrow \psi(x))\to U=W).'$$ 
\end{theorem}

\begin{proof}
Working in a model $M^G$ of $\mathsf{MV}_T$ (see section \ref{MV}), let $U$ and $W$ be worlds defined by $\varphi$ and $\psi$, respectively. Since $U$ and $W$ are transitive and contain the same ordinals, it will be sufficient to  show that $U\subseteq W$ by showing $V_\alpha^U\subseteq W$ by induction on $\alpha$. This is clear for $\alpha =0$, and also clear for $\alpha$ a limit ordinal provided it holds for all ordinals less than $\alpha$.
So suppose $V_\alpha^U \subseteq W$ and let us show $V_{\alpha+1}^U\subseteq W$. Note that $V_\alpha^U\subseteq W$ implies $V_\alpha^U \in W$. This is because
$$x \in V_\alpha^U \,\,\mbox{  iff  } \,\,W \models t('rank(v)\!<\!\alpha \mbox{  in   the unique world defined by } \varphi')[x]$$
Hence, $V_\alpha^U$ is definable in $W$ with parameter $\alpha$.

Now let $Y\subseteq V_{\alpha}^U$, $Y\in U$.  We must check that $Y\in W$. As in Proposition \ref{homo}, we can find
$\gamma$ large enough so that
$$W^{Coll(\omega,\gamma)}\models \mbox{'I am a forcing extension of $X$, for some world $X$ which }$$ $$\mbox{\hspace{1.8cm}is definable by }  \varphi  \mbox{  in the multiverse generated by me}.'$$
It follows that $U\subset W[H]$ for all $H$  $Coll(\omega,\gamma)$-generic over $W$. So
$Y\in W[H]$ for all such $H$.  Letting $\dot{Y}\in W$ be a $Coll(\omega,\gamma)$-name for $Y$, we have that $Y$ is definable in $W$ as  the set of all $x\in V_\alpha^U$ which are forced by $Coll(\omega, \gamma)$ to belong to $\dot{Y}$. Hence, $Y\in W$.
\end{proof}

It follows that if there is a definable world $U$ in the multiverse, then it is unique and, by Proposition \ref{homo},\ is contained in all of $\mathsf{MV}$'s worlds.\footnote{\cite{steel2014}, p. 168ff.} Clearly, since the core, if it exists, is definable, the core exists if and only if there is a definable world.

\subsection{The Core is the Mantle of $V$}

We proceed to introduce the fundamental result, due to Toshimichi Usuba, which shows that, under the assumption of sufficiently strong LCs, $V$ has a \textit{smallest} ground. 

\begin{theorem}[\cite{U:ECM}, p. 72]\label{usuba}

Suppose there exists an extendible cardinal. Then the mantle is a ground of $V$. In fact if $\kappa$ is extendible, the $\kappa$-mantle of $V$ is its smallest ground.\footnote{For a cardinal $\kappa$, the $\kappa$-mantle is the intersection of all $\kappa$-grounds, that is of all grounds $W$ of $V$, such that there exists a $\mathbb{P} \in W$ of size less than $\kappa$, and generic $G \subseteq \mathbb{P}$, such that $V=W[G]$.}

\end{theorem}

\noindent
Now, since \textit{all} grounds are \textit{downwards-directed},\footnote{See footnote \ref{DDG}.} by using Theorem \ref{usuba}, we can now prove the following:

\begin{prop}[Existence of the Core of the Multiverse of $\mathsf{MV}_T$]\label{coreexists}
Let $T$ be $\mathsf{ZFC}$+`there exists a proper class of extendible cardinals'. Then the $\mathsf{MV}_T$ axioms imply that the multiverse has a core, which is the mantle (and a ground) of every world in the multiverse.\footnote{Note the assumption, in Theorem \ref{coreexists}, that there exists a \textit{proper class}, not just \textit{one} extendible in all worlds of $\mathsf{MV}$. The reason for that is that a proper class of extendibles is needed, if one wants to preserve \textit{extendibles} in the set-forcing extensions of $V$ which count as worlds.}
\end{prop}

\begin{proof}
By Theorem \ref{usuba}, the mantle $\mathbb{M}_W$ of every world $W$ of the multiverse is a ground, hence by Axiom 5 of $\mathsf{MV}$, it is also a world. Now suppose $U_0$ and $U_1$ are worlds. By Amalgamation, they are also grounds of some world $W$ containing them. Hence, since the grounds are downwards-directed, $\mathbb{M}_{U_0}=\mathbb{M}_{U_1}$. It follows that all the worlds of the multiverse have the same mantle $\mathbb{M}$, which is the intersection of all the worlds. Thus, $\mathbb{M}$ is definable by the formula
$$\forall U (U\mbox{ is a world}\to x\in U).$$
Hence by Theorem \ref{defworld} the mantle $\mathbb{M}$ is the core of the multiverse.
\end{proof}

\subsection{The Strength of the Large-Cardinal Assumptions}

By Proposition \ref{coreexists}, if $M$ satisfies that there exists an extendible cardinal, then $M^G$ has a \emph{core}. It is legitimate to ask whether the large-cardinal assumption in Theorem \ref{usuba} might be weakened to the level of a LC lower in the consistency strength hierarchy; more generally:

\begin{question}\label{supercompact}

For what choices of $T$ does $\mathsf{MV}_T$ prove the existence of a core of its multiverse?

\end{question}

The following theorem shows that the existence of a proper class of supercompact cardinals is not sufficient to prove the existence of the core, thus suggesting a potential threshold for the choice of $T$ mentioned by Question \ref{supercompact}. 

\begin{theorem}\label{threshold}
Assume there is a proper class of supercompact cardinals. Then there is a class forcing notion which forces that there exists a proper class of supercompact cardinals and  there is no core.
\end{theorem}

\begin{proof}
First, force with a class-forcing iteration with Easton support in order to make every supercompact $\kappa$  indestructible by $<\!\!\kappa$-directed-closed forcing (see \cite{U:AA}), followed by  the Jensen's class-forcing iteration that forces the GCH. Standard arguments show that all the supercompact cardinals are preserved and no new supercompact cardinals are created. Over the generic extension, make again all supercompact cardinals $\kappa$ indestructible by $<\kappa$-directed-closed forcing (which preserves the GCH and does not create new supercompact cardinals) and call the resulting model $V[G]$. Now force with a class-forcing iteration $\mathbb{P}$ with Easton support that forces a version of the  Continuous Coding Axiom (CCA), similarly as in \cite{reitz2007}, which codes every set of ordinals  proper-class-many times into the power-set function on the class $S$ of supercompact cardinals that are not limit of supercompact cardinals. Again, standard arguments show that all supercompact cardinals in $S$ are preserved and no new supercompact cardinals are created. Let $V[G][H_0]$ be this model. Finally, force over $V[G][H_0]$ with the class-forcing product $$\mathbb{Q}=\prod_{\kappa \in ORD}\mathbb{Q}(\kappa)$$
where $\mathbb{Q}(\kappa)$ is the forcing for adding a Cohen subset of $\kappa$, if $\kappa$ is in $S$, and is the trivial forcing otherwise. Let the resulting  model be $V[G][H_0][H_1]$. We claim  that supercompact cardinals in $S$ are preserved. So  let $\kappa\in S$. 
Working in $V[G][H_0]$, note first that $\mathbb{Q}$ factors as $\mathbb{Q}_{\kappa} \times \mathbb{Q}_{[\kappa +1 , Ord)}$. 
%If $\kappa$ has an immediate predecessor $\kappa
%^-$ in $S$, then $\mathbb{Q}_\kappa$ is equivalent to $\mathbb{Q}_{\kappa^- +1}$, which %has size $\kappa^-$, and therefore standard arguments show that it preserves the %supercompactness of $\kappa$. 
Also note that, since $\kappa$ is not a limit of supercompact cardinals,  ${\rm{sup}}(S\cap \kappa)<\kappa$. So, $\mathbb{Q}_{\kappa}$ factors as $\mathbb{Q}_{{\rm{sup}}(S\cap \kappa)}\times \mathbb{Q}(\kappa)$, where $\mathbb{Q}(\kappa)$ is the forcing that adds a Cohen subset of $\kappa$. The product  $\mathbb{Q}_{{\rm{sup}}(S\cap \kappa)}\times \mathbb{Q}(\kappa)$ is equivalent, as a forcing notion, to $\mathbb{Q}(\kappa)\times \mathbb{Q}_{{\rm{sup}}(S\cap \kappa)}$. Also, since $\mathbb{Q}_{{\rm{sup}}(S\cap \kappa)}$ has cardinality less than $\kappa$ and $\mathbb{Q}(\kappa)$ does not add new bounded subsets of $\kappa$, $\mathbb{Q}(\kappa)\times \mathbb{Q}_{{\rm{sup}}(S\cap \kappa)}$ is equivalent to $\mathbb{Q}(\kappa)\ast \mathbb{Q}_{{\rm{sup}}(S\cap \kappa)}$. Now, since $\mathbb{Q}(\kappa)$ is $<\kappa$-directed closed it preserves the supercompactness of $\kappa$, and then so does $\mathbb{Q}_{{\rm{sup}}(S\cap \kappa)}$ subsequently, since it has cardinality less than $\kappa$.   Moreover, $\mathbb{Q}_\kappa$ forces that $\mathbb{Q}_{[\kappa +1 , Ord)}$ is $<\kappa$-directed-closed, and therefore it also preserves the supercompactness of $\kappa$. Since $\kappa^{<\kappa}=\kappa$ for every $\kappa \in S$, the forcing $\mathbb{Q}$ preserves cardinals and the power-set function. Now we may argue similarly as in \cite{reitz2007} to show that every ground of $V[G][H_0][H_1]$ has a proper subground, hence there is no core.\footnote{\cite{reitz2007}, Theorem 3.9, p. 1309.} For suppose $W$ is a ground of $V[G][H_0][H_1]$. Let $\mathbb{P}$  be a poset in $W$ and $g$ a $\mathbb{P}$-generic filter over $W$ such that $V[G][H_0][H_1]=W[g]$. Since the two models $V[G][H_0][H_1]$ and $W$ agree on the values of the power-set function for a tail of elements of $S$, as well as on statements of the form 
``$\kappa$ is the $\alpha$-th element of $S$" for a tail of cardinals $\kappa$ in $S$, every element of $V[G][H_0]$ is coded into the power-set function of $W$ on $S$, hence $V[G][H_0]\subseteq W$. Now let $\kappa$ in $S$ be greater than $|\mathbb{P}|$, and let $h_\kappa$ be the Cohen generic subset of $\kappa$ added by $H$. Since $V[G][H_0]\subseteq W$, every condition of the forcing $\mathbb{Q}(\kappa)$, therefore every bounded subset of $h_\kappa$, belongs to $W$. Hence, since the pair $\langle W, W[h]\rangle$ satisfies the $\kappa$-approximation property, $h_\kappa \in W$. Thus, writing $H_1$ as the product $H_1^{\leq \kappa}\times H_1^{>\kappa}$, we have that $V[G][H_0][H^{>\kappa}]\subseteq W$, and $V[G][H_0][H^{>\kappa}]$ is a ground of $W$. Now, if $\kappa_0 <\kappa_1$ are the first two members of $S$ greater than $|\mathbb{P}|$, then $V[G][H_0][H^{>\kappa_1}_1]$  is a proper ground of $V[G][H_1][H_1^{>\kappa_0}]$, hance also a proper ground of $W$.
\end{proof}

Therefore, although no formal evidence may currently be provided, it is reasonable to conjecture, that, as a consequence of Theorem \ref{threshold}, no LCA of consistency strength lower than that of `there exists an extendible cardinal' will be sufficient to prove the existence of the core. However, more work is needed to shed further light on the issue.

\section{The Core: CH and Forcing Axioms}

\subsection{The Core and CH}

We know that, under certain assumptions on $T$, the core of the multiverse of $\mathsf{MV}_T$ exists, so it makes sense to find out what features it has and, based on these, what it may reveal to us about such \textit{undecidable} statements as CH. In particular, we would like to address the following question:

\begin{question}\label{questCH}
Does $\mathsf{MV}_T$ prove that the core satisfies the CH, where $T=\mathsf{ZFC}+$`there exists a class of extendible cardinals'? 
\end{question}

The question may, in fact, extend to any other statement $\varphi$ which is not decided by $\mathsf{ZFC}$. In particular, for each such $\varphi$, one may ask whether the core implies the truth or falsity of $\varphi$. 

The results which follow provide answers to such questions, and, overall, starkly expose the \textit{indeterminacy} of the core. We, first, outline the mathematical strategy for CH. 

If the core exists, then it satisfies GA, namely, it does not have any proper grounds. \cite{reitz2007} proved that every model of $\mathsf{ZFC}$ has a class-forcing extension which preserves any desired $V_\alpha$, satisfies V=HOD, and is a model of $\mathsf{ZFC}$+GA.\footnote{\cite{reitz2007}, Theorems 3.5, and 3.10, pp. 1308-1310.} By results in \cite{bagaria-poveda2021}, the class-forcing used to obtain the model preserves extendible cardinals. So, starting with a model $M$ satisfying $T=\mathsf{ZFC}+$`There is a proper class of extendible cardinals' we may class-force over it to obtain a model $M[H]$ of $T$ satisfying GA and, e.g., $\neg$CH. And then by forcing with $Coll(\omega, <Ord)^{M[H]}$ over $M[H]$ we obtain a model of $\mathsf{MV}_T$ whose core is $M[H]$ and satisfies $\neg$CH. 

Now we show how this strategy can be extended to all $\Sigma_2$ set-forceable statements $\varphi$: 

\begin{theorem}\label{sigma2}
Let $\varphi$ be a $\Sigma_2$-statement (with parameters) that can be forced by set forcing. Assume there is a proper class of extendible cardinals. Then in some class-forcing extension of $V$ that preserves extendible cardinals the statement $\varphi$ holds in the core of $\mathsf{MV}$'s multiverse built around the extension itself.
\end{theorem}

\begin{proof}
First force $\varphi$ by set forcing. In the forcing extension let $\kappa \in C^{(1)}$ be such that $V_\kappa \models \varphi$. Then force the GCH above $\kappa$, in the usual way using class forcing, so that $V_\kappa$ is not changed. As in \cite{reitz2007}, we may further class-force to get a model of GA, so that $V_\kappa$ is still unchanged. By results contained in \cite{tsaprounis2014} and  \cite{bagaria-poveda2021}, both class-forcing notions  preserve extendible cardinals.\footnote{\cite{tsaprounis2014}, Lemma 4.3, p. 113 and \cite{bagaria-poveda2021}, Theorem 5.4, p. 15, and Theorem 7.5, p. 23-4.} Hence, since the final extension $M$ satisfies the GA, if $G$ is $Coll(\omega , <Ord)^M$-generic over $M$, then  the core of the $\mathsf{MV}$ multiverse given by $M^G$, which exists because in $M$  there exists a proper class of extendible cardinals, is $M$ itself. Also, since $V_\kappa$ has not been changed and satisfies $\varphi$, in $M$ we have that $V_\kappa \models \varphi$ and, moreover, $\kappa \in C^{(1)}$, because being in $C^{(1)}$ is characterized by being uncountable and satisfying $V_\kappa =H_\kappa$. Hence, $M\models \varphi$.
\end{proof}

\begin{corollary}
If $T=\mathsf{ZFC}+$`$\mbox{There exists a proper class of extendible cardinals}$' is consistent, then so is $\mathsf{MV}_T$ plus that the core  satisfies the CH, or $\neg$CH, with any possible value of the size of the continuum. 
\end{corollary}

Besides the CH and $\neg$CH there are many other relevant $\Sigma_2$ statements that are set-forceable, and therefore consistently hold in the core of the $\mathsf{MV}_T$ multiverse. We give next a couple more  examples. 

\begin{corollary}
If $T$ as above is consistent, then so is $\mathsf{MV}_T$ plus that the core  satisfies $\square_{\omega_2}$, or its negation. 
\end{corollary}

\begin{proof}
$\square_{\omega_2}$ is a forcible, by countably-directed and $\omega_1$-strategically closed forcing, $\Sigma_1$ statement, with $\omega_1$ and $\omega_2$ as parameters. Also, its negation is forced by $Col(\omega_1, <\kappa)$, with $\kappa$ a Mahlo cardinal.
\end{proof}

The well-known forcing axiom $\MA_{\aleph_1}$ is also equivalent to a $\Sigma_2$ statement, with $\aleph_1$ as a parameter. Thus,  Theorem \ref{sigma2} yields the following: 

\begin{corollary}\label{MAaleph1}
If $T$ as above is consistent, then so is $\mathsf{MV}_T$ plus that the core  satisfies $\MA_{\aleph_1}$. 
\end{corollary}

\subsection{The Core and Strong Forcing Axioms}

Corollary \ref{MAaleph1} helps us to introduce, in fuller generality, the examination of the behaviour of the core with respect to forcing axioms. Now, what has neatly come to surface is that the core of $\mathsf{MV}_T$, where $T$ may, and even may not, have extendible cardinals, is also consistent with \textit{strong} forcing axioms. Corollary 3.8 of \cite{reitz2007} already proves that the Proper Forcing Axiom (PFA) is consistent with the GA.\footnote{\cite{reitz2007}, p. 1308.} We extend this result also to MM, MM$^{++}$, etc. 

Recall that MM$^{++}$ states that for every poset $\Pe$ that preserves stationary subsets of $\omega_1$, every collection $\{ D_\alpha :\alpha <\omega_1\}$ of dense open subsets of $\Pe$, and every collection $\{ \sigma_\alpha :\alpha <\omega_1\}$ of $\Pe$-names for stationary subsets of $\omega_1$, there exists a filter $G\subseteq \Pe$ that is generic for $\{ D_\alpha :\alpha <\omega_1\}$ (i.e., $G\cap D_\alpha \ne \emptyset$, all $\alpha$), and such that $\sigma_\alpha [G]$ is a stationary subset of $\omega_1$, all $\alpha$. It is well-known that MM$^{++}$ can be forced, assuming the existence of a supercompact cardinal.

\begin{theorem}
If $\mathsf{ZFC}$ plus the existence of a supercompact cardinal is consistent, then so is MM$^{++}$ plus GA.
\end{theorem}

\begin{proof}
Let $V$ satisfy $\mathsf{ZFC}$ plus the existence of a supercompact cardinal. Force over $V$ to obtain a model of $\mathsf{ZFC}$ plus MM$^{++}$. Call this model $V[G]$. Then force with an $\omega_2$-directed-closed $ORD$-length iteration $\Pe$ over $V[G]$ to obtain a model $V[G][H]$ of GA (as in \cite{reitz2007}).  We claim that MM$^{++}$ holds in $V[G][H]$. For suppose  $\tau$ is a $\Pe$-name for a poset that preserves stationary subsets of $\omega_1$,  $\{ \dot{D}_\alpha : \alpha <\omega_1 \}$ is a $\Pe$-name for a collection of dense open subsets of $\tau$, and $\{ \dot{\sigma}_\alpha :\alpha <\omega_1\}$ is a $\Pe$-name for a collection of $\tau$-names for stationary subsets of $\omega_1$. Let $\lambda$ be a large enough cardinal such that $\tau$ and $\langle \dot{D}_\alpha : \alpha <\omega_1 \rangle$ and $\{ \dot{\sigma}_\alpha :\alpha <\omega_1\}$ are $\Pe_\lambda$-names. Since $\Pe_\lambda$ is $\omega_2$-directed-closed, arguing similarly as in \cite{larson2000}, we can show that $\Pe_\lambda$ preserves MM$^{++}$. 

We claim that in $V[G][H_\lambda]$, the poset $\tau[H_\lambda]$ preserves stationary subsets of $\omega_1$: For suppose $S\subseteq\omega_1$ is stationary and $\dot{C}$ is a $\tau[H_\lambda]$-name for a club subset of $\omega_1$. Since the remaining part of the iteration does not add any new subsets of $\omega_1$, $S$ is also stationary in $V[G][H]$. Moreover, since being a club is absolute for transitive models, and every $\tau[H_\lambda]$-generic filter over $V[G][H]$ is also $\tau[H_\lambda]$-generic over $V[G][H_\lambda]$, we have that in $V[G][H]$, $\dot{C}$ is a $\tau[H_\lambda]$-name for a club subset of $\omega_1$. Hence, since in $V[G][H]$ the poset $\tau[H_\lambda]$ preserves stationary subsets of $\omega_1$, we have that $\Vdash_{\tau[H]}``\check{S}\cap \dot{C}\ne \emptyset"$. But the latter is absolute for transitive models, and so it holds in $V[G][H_\lambda]$.

Moreover, in $V[G][H_\lambda]$, $\{\dot{\sigma}_\alpha[H_\lambda] :\alpha <\omega_1\}$ is a collection of $\tau$-names for stationary subsets of $\omega_1$. Since MM$^{++}$ holds in $V[G][H_\lambda]$, there exists a filter $F\subseteq \tau[H_\lambda]$ that is generic for the collection $\{ \dot{D}_\alpha [H_\lambda]:\alpha <\omega_1\}$, and such that $\dot{\sigma}_\alpha[H_\lambda][F]$ is stationary. By absoluteness this is also true in $V[G][H]$. This shows that MM$^{++}$ holds in $V[G][H]$.
\end{proof}

It follows from the theorem above that if $T=\mathsf{ZFC}$+`there exists a supercompact cardinal' is consistent, then so is that the core of $\mathsf{MV}_T$ satisfies MM$^{++}$. Also, since the class-iteration $\Pe$ that forces the GA preserves extendible cardinals,\footnote{See, again, \cite{bagaria-poveda2021}, Lemma 7.4 and Theorem 7.5.} if the theory $T=\mathsf{ZFC}+$`there exists a proper class of extendible cardinals' is consistent, then so is $\mathsf{MV}_T$ plus that the core satisfies MM$^{++}$. Moreover, in both cases, by results in \cite{aspero-schindler2019}, the core also satisfies Woodin's $(\ast)$ axiom.

\section{Probing Steel's Programme: The Core as Ultimate-$L$}

Let's take stock. All the results illustrated in sections 3 and 4 show that, assuming a proper class of extendible cardinals, the core of the multiverse of $\mathsf{MV}$ exists, but is still a highly \textit{indeterminate} object. In particular, we have seen that the core may satisfy any of the strongest known forcing axioms, all of which imply that the continuum has size $\aleph_2$, and it may also satisfy CH, and, for that matter, any other $\Sigma_2$ set-forceable statement, with parameters  (Theorem \ref{sigma2}). 

Two immediate considerations are in order. The first is that, based on the provable, under certain assumptions, existence of the core of the multiverse of $\mathsf{MV}$, the Universist may now be licensed to shift to the Core Universist view that $V$ \textit{is} the core, by settling on a theory, in $\mathcal{L}_{\in}$, which states this explicitly, that is, a theory which contains $V=\mathcal{C}$ as an axiom (henceforth, $\mathcal{C}$ will be our designated symbol for the core); in the next section we will examine this possibility in more detail.\footnote{See also a  discussion of this in \cite{maddy-meadows2020}, p. 146.} The second one is that, as a consequence of the persistent indeterminacy of the core also in the presence of LCs, even if the Core Universist decided to settle on $V=\mathcal{C}$ as the correct extension of $\mathsf{ZFC}$+LCs, she wouldn't still be able to fix the features of the core, and thus, the truth-value, in $\mathcal{L}_{\in}$, of the undecidable statements. 

So, it is time for Steel's Programme to kick in, and, in the next subsections, we will examine its execution and assess its prospects.

\subsection{Ultimate-$L$ and $\mathsf{MV}$}

In recent years, as a consequence of the far-reaching developments of the \textit{inner model programme}, the possibility of the existence of a canonical inner model for a  \textit{supercompact} cardinal is emerging. Woodin has called such a model Ultimate-$L$, insofar as this model would incorporate all features of an $L$-like inner model, as well as \textit{all} LCs, thus leading to the completion of the inner model programme itself.\footnote{See footnote \ref{Woodin}. For an accessible overview of the inner model programme, see \cite{jensen1995}; for technical results, \cite{devlin1984}, or \cite{kanamori2009}.} If such a model exists, then one could say that it represents an `optimal' approximation of $V$, something which would justify viewing the axiom $V$=Ultimate-$L$ as the most natural extension of $\mathsf{ZFC}$, and \cite{steel2014} has precisely taken into account such a possibility.\footnote{Although, in that work, Woodin's Ultimate-$L$ Conjecture is not directly addressed. See further, subsection 5.2.} 

In this subsection, we will first be concerned with mathematical results about Ultimate-$L$ \textit{and} the core and, in the next one, we will articulate the Core Universist position based on these results. We start with  Woodin's definition of the axiom $V$=Ultimate-$L$.

\begin{definition}[$V$=Ultimate-$L$]
The axiom $V$=Ultimate-$L$ asserts:
\begin{enumerate}
    \item  There is a proper class of Woodin cardinals, \emph{and}
    
    \item For each $\Sigma_2$ sentence $\varphi$, if $\varphi$ holds in $V$, then there is a universally-Baire set $A\subseteq \mathbb{R}$ such that
    $$HOD^{L(A,\mathbb{R})}\models \varphi.$$
\end{enumerate}
\end{definition}

\noindent
Woodin has shown that $V$=Ultimate-$L$ implies the CH.\footnote{See \cite{woodin2017}, Introduction.} It also implies the Ground Axiom, i.e., that $V$ is not a set-generic extension of any inner model, and that $V$=HOD.

 Woodin has made the following conjecture. First, recall that an inner model $N$ is a \emph{weak extender model for the supercompactness of $\delta$}  if for every $\lambda \geq \delta$ there is a normal fine measure $\mathcal{U}$ on $\mathcal{P}_\delta  \lambda$ such that:
\begin{enumerate}
    \item $\mathcal{U}\cap N\in N$
    \item $\mathcal{P}_\delta \lambda \cap N \in \mathcal{U}$.
\end{enumerate}

\begin{conjecture}
[Woodin's Ultimate-$L$ Conjecture]\label{WoodinUlt}
Suppose that $\delta$ is an extendible cardinal. Then there is an inner model $N\subseteq HOD$ such that:
\begin{enumerate}
    \item $N$ is a weak extender model for the supercompactness of $\delta$.
    \item $N\models$`$V$=Ultimate-$L$'.
\end{enumerate}
\end{conjecture}

\noindent
Now, crucially for our purposes, if the Ultimate-$L$ Conjecture holds, then letting $T$ be the theory $\mathsf{ZFC}+$`There exists a proper class of extendible cardinals' we have that $\mathsf{MV}_T$ proves that the core has an inner model $\mathcal{V}$ which is contained in HOD, and satisfies  $V$=Ultimate-$L$. Moreover, by Woodin's Universality Theorem,\footnote{\cite{woodin2017}, p. 16, Theorem 3.26.} $\mathcal{V}$ satisfies that there exists a proper class of extendible cardinals. However, $\mathcal{V}$ need not be a world itself, but, if it is so, then $\mathcal{V}$ is the \textit{core}. So, here follows the first  result of this subsection:  

\begin{proposition}
 Let $T$ be $\mathsf{ZFC}+$`There exists a proper class of extendible cardinals'. Then $\mathsf{MV}_T$ proves that the following are equivalent:
\begin{enumerate}
    \item There exists a world satisfying `$V$=Ultimate-$L$'
    \item $\mathcal{C}\models$`$V$=Ultimate-$L$'.
\end{enumerate}
\end{proposition}

 \begin{proof}
Let $M^G$ be a model of $\mathsf{MV}_T$, and let $\mathcal{W}$ be a world in $M^G$ that satisfies $V$=Ultimate-$L$. Then $\mathcal{W}$ is a ground of some world. But $\mathcal{W}$ has no proper grounds, hence since $MV_T$ implies that the core exists and is the mantle, $\mathcal{W}$ must be contained in the core, and therefore it must be the core.
\end{proof}  

The following two propositions nail down the consistency of $V$=Ultimate-$L$ with $\mathsf{MV}_T$ (where $T=\mathsf{ZFC}$+`there exists a class of extendible cardinals'), assuming $V$=Ultimate-$L$'s own consistency.

\begin{proposition}\label{Ult-L}
 Let $T$ be $\mathsf{ZFC}+$`There exists a proper class of extendible cardinals'. If $T+$`$V=\mbox{Ultimate}$-$L$'  is consistent, then so is $\mathsf{MV}_T$ plus `$\mathcal{C}\models$ T+$V$=Ultimate-$L$'.
\end{proposition}

\begin{proof}
Let $M$ be a model of $T+$`$V=\mbox{Ultimate}$-$L$'. Then for every $G\subseteq Coll(\omega , <Ord)^M$-generic filter over $M$, $M^G$ satisfies `$\mathcal{C}\models V=$Ultimate-$L$', by the last Proposition.
\end{proof}

\begin{proposition}
Assume the Ultimate-$L$ Conjecture. If $T=\mathsf{ZFC}+$`There exists a proper class of extendible cardinals' is consistent, then so is $\mathsf{MV}_T+$`$\mathcal{C}\models V$=Ultimate-$L$'.
\end{proposition}

\begin{proof}
Let $M$ be a model of $T$. Let $\delta$ be the least extendible cardinal in $M$.  By the Ultimate-$L$ Conjecture, $M$ has an inner model $N$ satisfying $V=$Ultimate-$L$ and which is a weak extender model for the supercompactness of $\delta$. By the Universality Theorem, in $N$ there is a proper class of extendible cardinals. Now let $G\subseteq Coll(\omega , <Ord)^N$ be generic over $N$. Then $N^G$ is a model of $\mathsf{MV}_T$ plus $\mathcal{C}\models$ $``V=$Ultimate-$L"$.
\end{proof}

However, also the negation of $V$=Ultimate-$L$, as shown by the following proposition, is consistent with $\mathsf{MV}_T$: 

\begin{proposition}\label{nUlt-L}
If the theory $T=\mathsf{ZFC}+$`There exists a proper class of extendible cardinals'+`$V$=Ultimate-$L$' is consistent, then so is $\mathsf{ZFC}+$`There exists a proper class of extendible cardinals'+`$\, \Ce \not \models V$=Ultimate-$L$'.
\end{proposition}

\begin{proof}
Let $M$ be a model of $T$. Force GA by class forcing, as in \cite{reitz2007}, so that (by \cite{bagaria-poveda2021}) in the resulting model $N$ there is a proper class of extendible cardinals. If $G\subseteq Coll(\omega , <Ord)^N$ is generic over $N$, then $N^G$ is a model of $\mathsf{MV}_T$ plus $\mathcal{C}=N\not \models`V$=Ultimate-$L$'. 
\end{proof}

\subsection{The Role of Ultimate-$L$ within $\mathsf{MV}$}

From what Steel says in \cite{steel2014}, it seems reasonable to assume that he would encourage the adoption of $\mathsf{ZFC}$+LCs+$V$-Ultimate-$L$ as the Core Universist's ultimate theory. It is crucial to make clear how this could possibly happen. 

Steel declares $V$=Ultimate-$L$ to be the right axiom for the core for the following reasons:\footnote{\cite{steel2014}, p. 169.} (1.) it implies $V=\mathcal{C}$; (2.) it implies the existence of a `fine structure theory' for the core; (3.) it is consistent with all LCs.\footnote{\label{AxiomH}At the time of the writing of \cite{steel2014}, $V$=Ultimate-$L$ had yet to be formulated as a mathematical axiom. The axiom that Steel advocated in that work was, in fact, Woodin's Axiom $\mathsf{H}$, in turn, a version of $V$=HOD. Now, assuming $V$=Ultimate-$L$, then the Axiom $\mathsf{H}$ also holds. However, the reasons Steel adduces as evidence in favour of the Axiom $\mathsf{H}$ are exactly the same as he might adduce in favour of $V$=Ultimate-$L$.} But, as we know, in view of $\mathsf{MV}$'s own machinery and goals, these motivations may just be seen as auxiliary arguments for the acceptance of $V$=Ultimate-$L$, as our prospects to settle on $V$=Ultimate-$L$ as the right extension of $\mathsf{ZFC}$ entirely depend, by $\mathsf{MV}$'s own lights, on detecting its presence in $\mathcal{L}_{\mathsf{MV}}$. 

So we proceed to assess these prospects by resuming the narrative about the Core Universist's possible choices laid out at the beginning of section 5.\footnote{See also \cite{maddy-meadows2020}, section 6.}

As said there, the Universist might tentatively interpret the $\mathcal{L}_{\in}$-translate of the $\mathcal{L}_{\mathsf{MV}}$-sentence `the core exists' as suggesting (`practically implying') the view that $V$ \textit{is} the core. Syntactically, this would correspond to adopting the theory $\mathsf{ZFC}+$LCs+$V=\mathcal{C}$. So, our Universist could, provisionally, settle on such a theory as the ultimate extension of $\mathsf{ZFC}$. However, as we know, the core is a highly \textit{indeterminate} object, so $V=\mathcal{C}$ wouldn't be very informative. By contrast, if our Universist interpreted the $\mathcal{L}_{\in}$-translate of `the core exists' as suggesting that $V$=Ultimate-$L$, then she would be able to \textit{determinately} fix the features of the core, and would get $V=\mathcal{C}$ as a bonus. To be more precise, in the presence of such an interpretation, $\mathsf{ZFC}+$LCs+$V$=Ultimate-$L$ would prove that CH, as expressed in $\mathcal{L}_{\mathsf{\in}}$, is equivalent to $t($`CH holds in the core'), %and, if we add, as an axiom, to $\mathsf{MV}_T$ that the core satisfies $V$=Ultimate-$L$, then, for the Universist, CH would hold 
and this strategy may be extended to all $\varphi$'s in $\mathcal{L}_{\mathsf{\in}}$ which aren't decided by $T= \mathsf{ZFC}+$LCs.%\footnote{In order to prevent misunderstandings, we wish to stress that \[CH \leftrightarrow t(CH^{Ultimate-L}) \] doesn't express the fact that `CH is true in Ultimate-$L$', but the fact that, if $\mathsf{MV}_T$ implies that the core is Ultimate-$L$, and, consequently, CH, then, from the point of view of $\mathcal{L}_{\in}$, CH is equivalent to CH$^{Ultimate-L}$.}

%\[ \varphi \leftrightarrow t(\varphi^{Ultimate-L}) \] 

So far so good, but is this strategy really workable? In fact, as we would like to point out, there may be severe hindrances in the way of its execution. First of all, if one looks at the functioning of (Transl), the $\mathcal{L}_{\mathsf{MV}}$-sentence `the core exists', doesn't literally translate to the $\mathcal{L}_{\in}$-sentence `$V$ is the core', let alone to `$V$=Ultimate-$L$' (see Definition \ref{Transl}). Therefore, some less quick, and more subtle kind of reasoning might underlie the adoption of $V=\mathcal{C}$ on the Universist's part, and this may have to do with the definability of the \textit{core} in $\mathsf{MV}$. As we know, if $\mathsf{MV}$ has a definable world, then it has a unique definable world (cf. Theorem \ref{defworld}). Now, we know that the core is definable, and since there is just \textit{one} such object in $\mathsf{MV}$, $\mathsf{MV}$'s unique definable world is the core itself. Therefore, the Core Universist might thrive on considerations about the \textit{definability} of the core within $\mathsf{MV}$, and see these as strongly suggesting that $V=\mathcal{C}$. But even if this argument ultimately persuaded us to adopt $V=\mathcal{C}$, it wouldn't be sufficient to persuade us to adopt $V$=Ultimate-$L$, anyway.

%Secondly, we've seen that $\mathsf{MV}_T$, where $T=$ZFC+LCs, is both consistent with `$\mathcal{C} \models$Ultimate-$L$', and with its negation (Propositions \ref{Ult-L}, and \ref{nUlt-L}). Therefore, not only the existence of the core doesn't imply that $V$ is the core, but it doesn't imply that $V$=Ultimate-$L$ either. 

For a different strategy, we could think that what we'd need here to make the $\mathcal{L}_{\mathsf{\in}}$-translate of the $\mathcal{L}_{\mathsf{MV}}$-sentence `the core exists' suggest that `$V$=Ultimate-$L$', would be some strengthening of $T$ in the theory $\mathsf{MV}_T$. As we shall see, unfortunately, all such strengthenings will produce results which are not fully satisfactory. 

Two main cases are possible. First, take $T$ to be theory $\mathsf{ZFC}$+LCs+$V$= Ultima-te-$L$, and consider $\mathsf{MV}_T$. The theory could, potentially, do the required job, but it is inconsistent. The reason is, any model of such theory would violate Axiom 4 of $\mathsf{MV}$, as no forcing extension of $\mathsf{MV}$ could be a world satisfying $T$. 
A more modest strengthening of $T$, on the other hand, might suit our purposes. Take $T$ to be: $\mathsf{ZFC}$+LCs+`$\mathcal{C}\models V$=Ultimate-$L$', then the $\mathcal{L}_{\mathsf{\in}}$-translate of `the core is Ultimate-$L$' would be a lot closer to what we'd need to have at hand, but, still, the translation wouldn't automatically imply `$V$ is the core'. But there's also another trouble with this choice of $T$: the axiom `$\mathcal{C}\models V$=Ultimate-$L$' wouldn't, \textit{prima facie}, seem to be justified on the grounds of $\mathsf{MV}$'s very evidential framework, so the only reason why one would add it to $T$ in $\mathsf{MV}_T$ would have to do with other, `external', reasons, such as, (1.)-(3.) above, or our prior belief in the correctness of $V$=Ultimate-$L$, but now relying on such reasons, somehow, would beg the question of why the Core Universist should adopt $V$=Ultimate-$L$. 

But maybe there are ways to make `$\mathcal{C}\models$V=Ultimate-$L$' justified on the grounds of $\mathsf{MV}$'s very evidential framework. We proceed to examine arguments potentially to that effect in the next subsections. 

\subsection{Woodin's Argument for Ultimate-$L$}

The first argument derives from Woodin's own reflections on the philosophical aspects of Conjecture \ref{WoodinUlt} (Ultimate-$L$ exists).

In \cite{woodin2011b}, Woodin has proposed to construe the Inner Model Programme as the expression of the fundamentally \textit{non-formalistic} character of set theory. Woodin's Argument, as we shall call it, first introduces two different positions about set-theoretic truth, the Skeptic's and the Set-Theorist's. 

\begin{itemize}
    \item \textit{Skeptic}: set-theoretic theorems are truths about finitary objects (\textit{proofs}).

    \item \textit{Set-theorist}: set-theoretic theorems are truths about an \textit{existing} realm of mathematical objects. 
\end{itemize}

Now, Woodin argues that the dialectic between the Skeptic and the Set-Theorist reaches its climax on the issue of the \textit{consistency} of LCs: the Skeptic holds that the consistency of LCs is a purely \textit{finitistic} fact (Skeptic's Retreat), whereas the Set-Theorist believes that it depends on intuitions about universes which contain them, in particular, \textit{canonical inner models}. In order to further articulate the Set-Theorist's position, Woodin formulates the following principle:

\vspace{11pt}

\noindent
\textbf{Set-Theorist's Cosmological Principle (SCP)}. A Large Cardinal Axiom is \textit{consistent} if and only if there is an inner model which satisfies it; the prediction of its consistency is true, because LCAs are \textit{true}.

\vspace{11pt}

\noindent
Woodin also indicates concrete ways in which SCP might be disconfirmed and the Skeptic's Retreat be validated, for instance, by finding a \textit{proof} that the consistency (inconsistency) of one specific LCA, implies the consistency (or inconsistency) of \textit{all} LCAs.\footnote{This would be the case, as shown by Woodin, if one found a proof of the consistency of $\mathsf{ZF}$+`there exist Reinhardt cardinals', as the consistency of such theory would, in turn, imply the consistency of $\mathsf{ZFC}$+`there is a proper class of $\omega$-huge cardinals'. The latter, in turn, implies all LCA known to be consistent with AC, so by a single proof, we would have a proof of the consistency of all LCs compatible with AC; cf. \cite{woodin2011b}, pp. 456-57.} Woodin then proceeds to make considerations about the possibility that, if the Ultimate-$L$ Conjecture is true, then we may have a situation where the scenario evoked by the Skeptic's Retreat is not applicable.\footnote{If Ultimate-$L$ exists, then we would know what large cardinals are and what are not in it, and, thus, by the SCP, we would know what large cardinals are consistent (as Ultimate-$L$ is a generalisation of all inner models for each single LC). The prediction of the consistency of all LCs based on inner models would thus reduce to the prediction of the existence of an ultimate core model which has all of them, and which would overall provide evidence for a single prediction of consistency of all LCs; cf. \cite{woodin2011b}, p. 463.}
%Cf. \cite{woodin2011b}, pp. 452ff. As regards how the SCP might be disconfirmed, Woodin's text is rather terse; however, it seems plausible to conjecture that Woodin's misgivings are directed at the fact that a proof of the consistency/inconsistency of \textit{one} LCA using inner models, might ultimately generalise to a proof of the consistency of \textit{other}, even \textit{all}, LCAs, something which would trivialise the Set-Theorist's view about the structure of consistency predictions.}  

We shall entirely leave aside the issue of the plausibility of Woodin's Argument, and exclusively focus on the potential usefulness of the argument for our Core Universist. The argument revolves around the idea that our belief in the correctness of \textit{all} LCs is based on the belief that they are consistent, in turn, via SCP, that they have \textit{inner models}. Thus, Woodin's Argument practically suggests that, if one sees LCs as the `right' extensions of $\mathsf{ZFC}$, it is because one views them as consistent via \textit{inner models}. So, the reasoning goes, once one commits oneself to \textit{all} LCs, as the $\mathsf{MV}$ supporter does, then, there is a (thick) sense in which one also commits oneself to Ultimate-$L$, that is, to Conjecture \ref{WoodinUlt}. But then, if Ultimate-$L$ exists, it should be a \textit{definable} world and, thus, based on Theorem \ref{defworld}, it would be the \textit{core}; the correctness of `$\mathcal{C} \models$V=Ultimate-$L$' would thus finally be vindicated, seemingly, on $\mathsf{MV}$'s own evidential grounds. 

%So, Woodin's argument brings in metaphysical considerations about the nature of universes of set theory and \textit{truth}. 

However, the argument has two main problems for our $\mathsf{MV}$-based Core Universist.

The first one is that the evidential resources it invokes (truth, consistency predictions, etc.) are not, in fact, available to her. Steel's conception prides itself on not being dependent on any `metaphysics' of universes. Therefore, intuitions about the structure of models with LCs, which underlie Woodin's Argument and SCP, if argumentatively efficacious, may not bear on our Core Universist's acceptance of `$\mathcal{C}\models$Ultimate-$L$'. 

The other one is that, by using Woodin's Argument in the present context, then the Core Universist would become practically indistinguishable from the Classic Universist who supports $V$=Ultimate-$L$: even coming from slightly different backgrounds, both would, indeed, agree on the fact that there is one \textit{universe} of set theory (Ultimate-$L$) which is preferable to all others from the beginning, and it does not seem that the $\mathsf{MV}$ supporter's evidential framework would commit her to such a view.  

Let us now move on to explore another potential justificatory strategy. 

\subsection{An Extrinsic Argument}

A different argument is based on resuming Steel's reasons to adopt $V$=Ultimate-$L$ (1.)-(3.) stated in section 5.2: these, overall, suggest to the Core Universist that $V$=Ultimate-$L$ is a very successful axiom, which could be seen as being \textit{already} justified `extrinsically' -- insofar as practically \textit{all} undecidable statements are settled by it.  

More precisely, the Core Universist might use a form of `regressive' reasoning here: she will require $\mathsf{MV}_T$ to incorporate `$\mathcal{C} \models$V=Ultimate-$L$' in $T$, because she takes `$V$=Ultimate-$L$' to be an \textit{already} independently justified axiom. 

However, in section 5.2, we have already hinted at the inherent limitations of this strategy: $\mathsf{MV}$ would no longer be used as a means to indicate what bits of $\mathcal{L}_{\in}$ ought to be believed to be `meaningful' (to legitimately hold); on the contrary, it would be $\mathcal{L}_{\in}$ to provide us information about what specific $T$ should be chosen in $\mathsf{MV}_T$. In other terms, the Core Universist wanting to use this `extrinsic' argument would, ultimately, violate her own unbiasedness about all different theories expressed by $\mathsf{MV}$.  

Another, more general, concern is that $V$=Ultimate-$L$ isn't the only successful axiom she has at hand. In particular, since she knows that the core need not be Ultimate-$L$, and that $T$ in $\mathsf{MV}_T$ is consistent with  `$\mathcal{C} \not \models$V=Ultimate-$L$' (Theorem \ref{nUlt-L}), the Core Universist could ultimately settle on other, equally successful, axioms for the core. For instance, she might want to adopt `$\mathcal{C} \models$ MM$^{++}$', as MM$^{++}$ implies that $c=\aleph_2$, and is clearly able to settle many other undecidable statements. 

One final worry about the argument, which is worth mentioning, is that `success' wasn't really part of Steel's narrative concerning $\mathsf{MV}$ from the beginning, although, clearly, LCs may be interpreted as being very `successful' axioms. We do not want to delve into the full intricacies of the topic, but `success' may really be a very volatile criterion, which, although helpful, may not lead the Core Universist to make ultimate choices about the nature of $\mathsf{MV}$. 

\subsection{Summary}

Let's take stock. The progression from $\mathsf{ZFC}$ to $T=\mathsf{ZFC}$+LCs+`$\mathcal{C} \models$V=Ultimate-$L$', summarised in the table below, shows that ascending through \textit{interpretative power} (and \textit{consistency strength}) of theories, is, presently, insufficient to suggest to the Core Universist that $V$ is the core, or that $V$ is any specific `world', for instance, Ultimate-$L$. As we have seen, on the one hand, by adding further hypotheses to $T$ in $\mathsf{MV}_T$, one may get an inconsistent theory; on the other, one could make choices, such as the addition of `$\mathcal{C} \models$V=Ultimate-$L$', which, however, on $\mathsf{MV}$'s own evidential grounds, do not seem to be much justified. This, overall, leaves us with the following, somewhat unpalatable, dilemma: either to stay with $T=\mathsf{ZFC}$+LCs, in $\mathsf{MV}_T$ and, thus, view the core as fundamentally indeterminate, or move to a, globally, less justified strengthening of $T$, but finally get a \textit{determinate} core.  

\medskip

\begin{table}[h]

\footnotesize

\begin{tabular}{||c c c c||} 
 \hline
 Theory & Core (Existence) & Core (Determinacy) & Suggests $V=\mathcal{C}$ \\ [0.5ex] 
 \hline
 ZFC & No & / & /\\
 \hline
 ZFC+LCs & Yes & No & Yes?\\
 \hline
 ZFC+LCs+$V$=Ultimate-$L$ & \textbf{Inconsistent!} &  & \\
 \hline
 ZFC+LCs+`$\mathcal{C} \models V$=Ultimate-$L$' & Yes & Yes & Yes?\\
 \hline
 ... & ... & ... &...\\ [1ex] 
 \hline

\end{tabular}

\caption{Behaviour of $\mathsf{MV}_T$ with respect to the core for different choices of $T$.}

\end{table}

\medskip

\normalsize

\section{Concluding Remarks and Ways Forward} 

Based on what one can prove in the theory $\mathsf{MV}_T$, where $T=\mathsf{ZFC}$+`there exists a proper class of extendible cardinals', and further potential strengthenings, we have assessed the prospects of what we have called Steel's Programme, and of the corresponding philosophical position that we have called Core Universism. 

Our tentative conclusion is that there might still be a long way to go before the view that the core is Ultimate-$L$ and, consequently, that $V$=Ultimate-$L$ is the ultimate axiom for the Core Universist, gets fully validated; meanwhile, as is clear, $V$=Ultimate-$L$ might ultimately be accepted on entirely different grounds. 

%Although the theory has already proved to be a nice mathematical environment to assess the strength of several hypotheses, further work is needed to fully assess the prospects of $\mathsf{MV}$. Therefore, 

We conclude the paper with suggesting some possible future scenarios for the investigation of the `core hypothesis' within $\mathsf{MV}$: 

\vspace{11pt}

\noindent
\textbf{First Scenario}. Stronger, hitherto unknown, LCs will, `more determinately', settle the features of the core.

\vspace{11pt}

\noindent
The idea, here, is that further ascending through consistency strength in an ideally `richer' large-cardinal hierarchy will help settle the features of the core. Clearly, at present, nobody could possibly foresee whether there will be concrete developments in this direction, but note that Usuba's key theorem (Theorem \ref{usuba}) has set an interesting precedent for results in this area which were widely unexpected.\footnote{We thank an anonymous referee for pointing out that, if one new type of LCA implied that CH holds in the core of the multiverse of $\mathsf{MV}$, then it would already directly imply CH, but this fact, then, would rather help advocate Classic than Core Universism.}

\vspace{11pt}

\noindent
\textbf{Second Scenario}. The `core hypothesis' will ultimately be taken to just be a theoretical tool to foster and study the interplay between \textit{multiverse} and \textit{universe thinking} while dealing with set-theoretic incompleteness. 

\vspace{11pt}

\noindent
The scenario above construes Steel's efforts as going in the direction of clarifying, not prescribing solutions for, the issues of whether there is a core universe and of what the core should be like. In this scenario, the core hypothesis would not, \textit{per se}, suggest a unique course of action for Universists, yet may be taken to be informative on what resources and additional hypotheses are needed if one wants to settle the undecidable statements.\footnote{This interpretation is, arguably, already inherent in Steel's declared goal of using $\mathcal{L}_{\mathsf{MV}}$ `to trim the syntax of $\mathcal{L}_{\in}$' and `thus avoiding asking pseudo-questions'. Cf. \cite{steel2014}, p. 168.}

\vspace{11pt}

\noindent
\textbf{Third Scenario}. The issue of the existence of a `preferred universe' will be declared to be \textit{insoluble} on purely proof-theoretic grounds (that is, by just focussing on the relationship between \textit{multiverse language} and \textit{language of set theory}), whilst further conceptual \textit{resources} will ultimately be seen as fundamental to solve it.

\vspace{11pt}

\noindent
As we have seen, Woodin's Argument, for instance, provides us with alternative resources to solve the issue of what counts as a `preferred' universe. Now, one could conjecture that it will turn out to be inevitable to resort to this kind of arguments to meet the Core Universist's requirements.

In the paper, we have also addressed and, by our lights, made more transparent, several features, both mathematical and philosophical, of $\mathsf{MV}$: in particular, the proof of the existence of the core, its persistent indeterminacy over different choices of $T$ for $\mathsf{MV}_T$, the mathematics of Ultimate-$L$ and $\mathsf{MV}$, the justifiability and role of LCs. Now, further mathematical work is, already at this stage, needed in order to further assess $\mathsf{MV}$'s hypotheses. In particular, ideally, much more should be known about:

\begin{itemize}
    \item The status of the Ultimate-$L$ Conjecture
    
    \item The prospects of a unified account of LCs 
    
    \item A better understanding of the notion of `canonicity' with respect to (alternative) models of set theory
\end{itemize}

We expect that further enlightenment on these issues will also carry with itself a more detailed understanding of what the Core Universist may legitimately claim, based on the structure of the multiverse. But, as we have seen, definite answers to many issues concerning $\mathsf{MV}$, some confirming, others disconfirming, the Core Universist's expectations, can already be provided.    

\pagebreak

\bibliographystyle{apalike} 
\bibliography{Bib1}

\pagebreak

\tableofcontents

\end{document}